\documentclass{amsart}
\usepackage{amsmath,amsthm}
\usepackage{amsfonts,amssymb}

\usepackage{enumerate}

\newtheorem{thm}{Theorem}[section]
\newtheorem{cor}[thm]{Corollary}
\newtheorem{lem}[thm]{Lemma}
\newtheorem{prop}[thm]{Proposition}

\newtheorem{defn}[thm]{Definition}

\theoremstyle{remark}

\def\sph{\mathbb{S}^{d-1}}
\def\f{\frac}

\def\Bl{\Bigl}
\def\Br{\Bigr}

 \def\g{{\gamma}}
 \def\k{{\kappa}}
 \def\t{{\theta}}
 \def\l{{\lambda}}
 
 \def\o{{\omega}}
 \def\s{{\sigma}}
 \def\la{{\langle}}
 \def\ra{{\rangle}}

 \def\CD{{\mathcal D}}
 \def\CE{{\mathcal E}}
 
 \def\CH{{\mathcal H}}

 \def\CV{{\mathcal V}}
 
 \def\BB{{\mathbb B}}

 \def\RR{{\mathbb R}}
 \def\SS{{\mathbb S}}
 \def\TT{{\mathbb T}}
 \def\ZZ{{\mathbb Z}}
        
        \def\proj{\operatorname{proj}}

  \def\dist{\mathtt{d}}

\def\Bl{\Bigl}
\def\Br{\Bigr}
\def\f{\frac}
\def\({\Bigl(}
\def \){ \Bigr)}

\begin{document}

\title[]
{Uncertainty principle on weighted spheres, balls and simplexes}
\author{Yuan Xu}
\address{Department of Mathematics\\ University of Oregon\\
    Eugene, Oregon 97403-1222.}\email{yuan@math.uoregon.edu}
\thanks{The work was supported in part by NSF Grant DMS-1106113.}

\date{\today}
\keywords{uncertainty principle, unit sphere, reflection invariant weight, ball, simplex}
\subjclass[2000]{42B10, 42C10}

\begin{abstract}
For a family of weight functions $h_\k$ that are invariant under a reflection group, the uncertainty principle
on the unit sphere in the form of 
$$
   \min_{1 \le i \le d}  \int_{\sph} (1- x_i) |f(x)|^2 h_\k^2(x) d\s \int_{\sph}\left |\nabla_0 f(x)\right |^2 h_k^2(x) d\s \ge c
$$
is established for invariant functions $f$ that have unit norm and zero mean, where $\nabla_0$ is the spherical
gradient. In the same spirit, uncertainty principles for weighted spaces on the unit ball and on the standard 
simplex are established, some of them hold for all admissible functions instead of invariant functions. 
\end{abstract}

\maketitle

\section{Introduction} 
\setcounter{equation}{0}

In the form of the classical Heisenberg inequality, the uncertainty principle in $\RR^d$ can be stated as 
\begin{equation*}
     \inf_{a \in \RR^d} \int_{\RR^d} \|x - a\|^2 |f(x)|^2 dx \int_{\RR^d} |\nabla f(x)|^2 dx
          \ge \frac{d^2}{4} \left(\int_{\RR^d} | f(x)|^2 dx \right)^2,
\end{equation*}
where $\nabla$ is the gradient operator. The uncertainty principle has been studied extensively in various 
settings, see \cite{FS, Thanga} and the references therein. Recently,  in \cite{DaiX},  we established an 
uncertainty principle on the unit sphere. Let $\sph$ denote the unit sphere in $\RR^d$ and let $\nabla_0$ 
denote the spherical gradient on $\sph$. The uncertainty inequality in \cite{DaiX} takes the form 
\begin{equation} \label{UCinequality}
   \min_{e\in\sph}  \int_{\sph} (1- \la x, e \ra) |f(x)|^2 d\s \int_{\sph}\left |\nabla_0 f(x)\right |^2 d\s \ge   c_d 
    \left( \int_{\sph} \left |f(x)\right |^2 d\s \right)^2
\end{equation}
for $f \in L^2(\sph)$ satisfying $\int_{\sph} f(x)d\s (x) =0$, where $d\s$ is the surface measure and $c_d$ is a 
constant depending on the dimension $d$ only. Let $\dist(x,y) = \arccos \la x,y\ra$ denote the geodesic 
distance on the sphere. Then $1- \la x,y\ra = 2 \sin^2 \frac{\dist(x,y)}{2} \sim [\dist(x,y)]^2$, which shows 
that the \eqref{UCinequality} is a close analogue of the classical Heisenberg uncertainty principle. Furthermore,
the inequality \eqref{UCinequality} is shown to be stronger than an uncertainty inequality previously
known  in the literature \cite{Erb, NW, RV, Selig}. 

The purpose of the present paper is to establish analogues of \eqref{UCinequality} in several different 
settings. First of all, we consider the weighed space $L^2(h_\k^2, \sph)$ for a family of weight functions 
$h_\k$ of the form
$$
   h_\k(x) = \prod_{v \in R_+} |\la x, v \ra|^{\k_v}, \qquad \k_v \ge 0,
$$
that are invariant under a reflection group, where $R_+$ denotes the set of positive roots that defines
the reflection group and $\k_v$ is a a nonnegative multiplicity function defined on $R_+$ whose values 
are equal whenever reflections in positive roots are conjugate. In the case of the group $\ZZ_2^d$, the
simplest case, 
$$
h_\k(x) = \prod_{i=1}^d|x_i|^{\k_i}, \qquad \k_i \ge 0.
$$
In the setting of a general reflection group, the role of rotation group, 
under which $d\s$ is invariant, is replaced by $h_\k^2 d\s$ and the partial derivatives are replaced by the 
Dunkl operators, $\CD_1,\ldots, \CD_d$, which are first order differential-difference operators that commute 
with each other  \cite{Dunkl}. In particular, the gradient is replaced by $\nabla_{h}:= (\CD_1,\ldots, \CD_d)$ and the 
operator $\nabla_0$ is replaced by the spherical part of $\nabla_{h,0}$, which coincides with $\nabla_0$ 
on functions invariant under the reflection group. Secondly, there is a close relation between analysis on 
the sphere and analysis on the ball, which allows us to consider the setting of $L^2(W, \BB^d)$ for a family 
of weight functions $W$ on the unit ball $\BB^d$ of $\RR^d$, including the weight function 
$$
  W_{\k,\mu}(x) = \prod_{i=1}^d |x_i|^{2\k_i} (1-\|x\|^2)^{\mu-1/2}, \quad \k_i \ge 0, \, \mu \ge 0,
$$
and, in particular, the classical weight function $W_\mu(x) = (1-\|x\|^2)^{\mu -1/2}$ on the ball. Thirdly, a 
further relation between analysis on $\BB^d$ and that on the simplex 
$\TT^d = \{x \in \RR^d: x_i \ge 0, 1 -x_1-\cdots -x_d \ge 0\}$ allows us to study uncertainty principles on the 
simplex $\TT^d$ with respect to several families of weight functions, including the classical 
weight functions
$$
U_\k(x) = x_1^{\k_1} \cdots x_d^{\k_d} (1-x_1-\ldots-x_d)^{\k_{d+1}}, \quad x \in \TT^d.
$$ 
 
Our proof relies on various properties of differential and differential-difference operators on the sphere. 
The background and the basic results are reviewed in Section 2. Based on the Dunkl operators and the
analogous Laplacian defined by $\Delta_h = \CD_1^2+ \ldots  +\CD_d^2$, a rich analogue of classical harmonic 
analysis for the measure $h_\k^2 d\s$ has been developed (see \cite{DaiXbook, DX} and the references 
therein). It turns out, however, that the analogue of the spherical gradient, $\nabla_{h,0}$, has not been 
studied much in the literature. A detailed study of this operator and several other related operators is 
carried out in Section 3;  the results in this section will likely be useful for further study in weighted 
spaces on the sphere. In Section 4, we establish our uncertainty inequalities in the space $L^2(h_\k^2, \sph)$.
Our inequality holds for invariant functions under a general reflection group. In Section 5, we deduce 
uncertainty inequalities on the unit ball 
from those on the sphere; for the classical weight function $W_\mu$ on the ball, our inequality holds for all 
admissible functions, not just for invariant functions. Finally, in Section 6, we deduce uncertainty inequalities 
on the simplex from those on the ball; the inequality for the classical weight function $U_\k$ holds for all 
admissible functions. 

\section{Preliminary} 
\setcounter{equation}{0}

Throughout this paper, we denote by $\langle x,y\rangle$ and $\|x\|$ the usual Euclidean inner product 
and norm. For the formal inner product between a vector in $\RR^d$ and an $d$-tuple of operators,
say $T = (T_1,\ldots, T_d)$, we use the notation of dot product $x \cdot T$. For example, 
$x \cdot \nabla = \sum_{i=1}^d x_i \partial_i$. We also use the dot product for the former inner product
between two $d$-tuple of operators. 

\subsection{Ordinary differential operators} Let $\partial_j$ denote the $j$-th partial derivative
operator. The usual gradient operator and the Laplace operator are defined by, respectively, 
$$
  \nabla = (\partial_1,\ldots, \partial_d) \qquad \hbox{and} \qquad \Delta = \partial_1^2 + \ldots \partial_d^2.
$$
We can also write, symbolically, that $\Delta  = \nabla \cdot \nabla$ using the formal dot product. In 
spherical polar coordinates, $x = r \xi$ with $r > 0$ and $\xi \in \sph$, we have 
\begin{equation}\label{eq:nabla}
  \nabla = \xi \frac{d}{d r} + \frac{1}{r} \nabla_0 \quad \hbox{and} \quad \Delta = \frac{d^2}{dr^2}
    + \frac{d-1}{r}\frac{d}{dr} + \frac{1}{r^2} \Delta_0,        
\end{equation}
where $\nabla_0$ is the spherical gradient vector and $\Delta_0$ is the Laplace-Beltrami 
operator. It is well known that $\Delta_0$ has spherical harmonics as eigenfunctions. 
Furthermore, these two operators acted on $\xi$ variables and it is easy to see that 
$$
\xi \cdot \nabla_0 = 0, \qquad \xi \in \sph.
$$
Let us also mention that 
$$
   \Delta_0 = \nabla_0 \cdot \nabla_0
$$
holds. There is another family of differential operators that interact with $\nabla_0$ and $\Delta_0$. 
They are angular derivatives defined by 
$$
       D_{i,j}  = x_i \partial_j - x_j \partial_i, \qquad 1 \le i \ne j \le d.
$$
In terms of polar coordinates $(x_i, x_j) = r_{i,j} (\cos \t_{i,j}, \sin \t_{i,j})$ on the $(x_i,x_j)$ plane, 
the operator $D_{i,j}$ is the angular derivative $D_{i,j} = \partial / \partial \t_{i,j}$. These operators are 
infinitesimal operators of the regular representation of the rotation group and they are closely related 
to the spherical gradient $\nabla_0$ and Laplace-Beltrami opeator $\Delta_0$. Indeed, we have  
\begin{equation} \label{Delta=Dij2}
   \Delta_0 = \sum_{1 \le i< j \le d} D_{i,j}^2,
\end{equation}
and 
\begin{equation} \label{nabla=Dij}
  \nabla_0 f(x) \cdot \nabla_0 g(x) = \sum_{1\le i < j \le d} D_{i,j}f(x) D_{i,j}g(x). 
\end{equation}
Furthermore, it is known that, for $f, g \in C^1(\sph)$,  
\begin{equation} \label{Dij=Dij}
  \int_{\sph} D_{i,j} f(x) g(x) d\s(x)  = -  \int_{\sph} f(x) D_{i,j} g(x) d\s(x).
\end{equation}
Most of the properties in this subsection are classical. The proof of the last three displayed identities
can be found in \cite[Section 1.8]{DaiXbook}.
 
\subsection{Differential-difference operators}
Let $v$ be a nonzero vector in $\RR^d$ . The reflection $\s_v$ along $v$ is defined by
$$
     x \sigma_v := x - 2 \langle x,v \rangle v /\|v\|^2, \qquad x \in \RR^d.
$$
Let $G$ be a finite 
reflection group on $\RR^{d}$ with a fixed positive root system $R_+$. Then $G$ is a 
subgroup of the orthogonal group generated by the reflections $\{\sigma_v: v \in R_+\}$. 
Let $\kappa$ be a nonnegative multiplicity function $v \mapsto \kappa_v$ defined on $R_+$ 
with the property that $\kappa_u = \kappa_v$ whenever $\sigma_u$ is conjugate to 
$\sigma_v$ in $G$; then $v \mapsto \kappa_v$ is a $G$-invariant function.  Associated with 
each $\sigma_v$, we define an operator, also denoted by $\s_v$, by
$$
   \s_v f(x):= f(x \s_v). 
$$
We are now in a position to define the Dunkl operators introduced in \cite{Dunkl}.

\begin{defn} 
Let $v\mapsto \k_v$ be a multiplicity function associated with a finite reflection group $G$.
The Dunkl operators are defined by, for $1 \le j \le d$,
$$
  \CD_j : = \partial_j  +  \sum_{v \in R_+} \k_v v_j E_v, \qquad   E_v   : =   \frac{ I - \s_v}{\la x ,v\ra}. 
$$
where $v = (v_1,\ldots, v_d)$ and $I$ denotes the identity operator.  
\end{defn}

These are first order differential-diffenece operators and they satisfy a remarkable commuting
property, 
$$
  \CD_i \CD_j = \CD_j \CD_i, \qquad 1 \le i, j \le d.
$$
Thus, they can be regarded as extensions of the ordinary partial differential operators. An analogue of the
Laplace operator, called $h$-Laplacian, is defined by
$$
   \Delta_h = \sum_{i =1}^d \CD_i^2, 
$$
which shares many properties of the ordinary Laplacian.  In terms of ordinary differential operators
$\nabla$ and $\Delta$, the $h$-Laplacian is given explicitly by
$$
 \Delta_h = \Delta + \sum_{v\in R_+} \k_v \left(\frac{2\,   v \cdot \nabla}{\la v,x\ra} 
        - \|v\|^2 \frac{I - \s_v}{\la x, v\ra^2} \right).
$$ 
In spherical polar coordinates, the $h$-Laplacian satisfies the relation 
\begin{equation}\label{eq:Delta_h}
  \Delta_h =   \frac{d^2}{dr^2} + \frac{2\l_\k+1}{r} \frac{d}{dr} + \frac{1}{r^2} \Delta_{h,0},
\end{equation}
where $\Delta_{h,0}$ denotes the $h$-Laplace-Beltrami operator that acts only on $\xi$ and 
\begin{equation}\label{eq:lambda_k}
   \lambda_\k : = \g_\k+ \frac{d-2}{2} \quad \hbox{with} \quad \g_\k : = \sum_{v\in R_+} \k_v. 
\end{equation}
 
\subsection{h-harmonics and orthogonal expansions}
A homogeneous polynomial $Y$ on $\RR^d$ that satisfies $\Delta_{h,0}Y = 0$ is called an 
$h$-harmonic polynomial. The restriction of such polynomials on the unit sphere $\sph$ 
are called $h$-spherical harmonics, which are orthogonal with respect to 
the inner product of $L^2(\sph, h_\k^2)$, where
\begin{equation} \label{hk-weight}
h_\k(x):= \prod_{v \in R_+} |\langle x, v\rangle|^{\kappa_v}, \qquad  x \in \RR^d.
\end{equation}
More precisely, let $\CH_n^d(h_\k^2)$ denote the space of $h$-harmonics of degree $n$
in $\RR^d$; then for $Y_n \in \CH_n^d(h_\k^2)$ with $n = 0,1,2,\ldots$,
$$
  \int_{\sph} Y_n (x) Y_m(x) h_\k^2(x) d\s(x) = 0, \qquad n \ne m.
$$
The space of $\CH_n^d(h_k^2)$ consists of the eigenfunctions of the $h$-Laplace-Beltrami
operator, that is,
\begin{equation}\label{eigen-eqn}
     \Delta_{h,0} Y(x) = - n(n + 2 \l_\k) Y(x), \qquad Y \in \CH_n^d(h_\k^2). 
\end{equation}

Let us denote by $L^2(\sph, h_\k^2)$ the space of measurable functions $f$ for which 
$$
   \|f\|_{L^2(\sph, h_\k^2)} : = \left(\frac{1}{\o_d^\k} \int_{\sph} |f(x)|^2 h_\k^2(x) d\s(x) \right)^{1/2} 
$$
are finite, where $\o_d^\k$ denotes the normalization constant 
$$
\o_d^\k : =  \int_{\sph} h_\k^2(x) d\s(x).
$$
Let $\proj_n^\k:L^2(\sph, h_\k^2) \mapsto \CH_n^d(h_\k^2)$ be the orthogonal projection operator.
The standard Hilbert space argument shows that 
$$
   f (x) = \sum_{n=0}^\infty \proj_n^\k f(x), \qquad  f \in L^2(\sph, h_\k^2), 
$$ 
where the equality holds in the $L^2$ sense. Since $\proj_n^\k f \in \CH_n^d(h_\k^2)$, it follows
from \eqref{eigen-eqn} that 
\begin{equation}\label{eq:Delta-power}
    (- \Delta_{h,0})^r f(x) =  \sum_{n=1}^\infty (n (n+2\l_k))^r \proj_n^\k f(x)
\end{equation}
for $r$ being a positive integer. Furthermore, we can use \eqref{eq:Delta-power} as the 
definition of $(-\Delta_{h,0})^r$ for $r$ being any real number. 

\section{Relations of differential-difference operators}
\setcounter{equation}{0}

The Dunkl operators $\CD_j$ play the role of differential operators $\partial_j$ in the setting of reflection
group. We define an analogue of the gradient vector as 
$$
    \nabla_{h} := (\CD_1, \ldots, \CD_d). 
$$
The next lemma defines $\nabla_{h,0}$, which is an analogue of the spherical gradient. 

\begin{lem} \label{lem:nabla_h0}
In the spherical polar coordinates $x = r \xi  \in \RR^d$, where $r > 0$ and $\xi \in \sph$, we have 
\begin{equation}\label{eq:nabla_h}
      \nabla_h = \xi \frac{d}{d r} + \frac{1}{r} \nabla_{h,0},
\end{equation}
where $\nabla_{h,0}$, an analogue of the spherical gradient, satisfies
\begin{equation}\label{eq:nabla_h0}
    \nabla_{h,0} f(\xi) = \nabla_0  f(\xi)+ \sum_{v\in R_+} \k_v E_v f(\xi) v, \quad \xi \in \sph.
\end{equation}
\end{lem}

\begin{proof}
By definition, $\nabla_h = \nabla + E$, where $E= \sum_{v\in R_+} \k_v E_v v$. 
Since we evidently have 
$$
      E_v f(r \xi) = \frac{1}{r} \frac{ f(r \xi) - f(r \xi \s_v)}{\la v, \xi \ra},
$$
the stated results follow immediately from this identity and \eqref{eq:nabla}. 
\end{proof}

Despite extensive studies of Dunkl operators and associated harmonic analysis (see, for example,
\cite{DaiXbook, DX}), the operator $\nabla_{h,0}$ and its relation with $\Delta_{h,0}$ and the operators
$\CD_{i,j}$, which are analogues of $D_{i,j}$ to be defined later, have not been studied much in the
literature. In the rest of this section, we will study relations between these operators carefully. 

We start with a simple, yet important, observation on $\nabla_{h,0}$. 
\begin{lem} 
For $\xi \in \sph$, 
\begin{equation} \label{xi-na_h0}
         \xi \cdot \nabla_{h,0} f(\xi) = \sum_{v \in R_+} \k_v \big[ f(\xi) - f(\xi \s_v) \big]. 
\end{equation}
\end{lem}

\begin{proof}
The definition shows immediately that  $\la \xi, v\ra E_v f(\xi) = f(\xi) - f (\xi \s_v)$. 
Hence, by $\xi \cdot \nabla_0 f(\xi) =0$, \eqref{xi-na_h0} follows from \eqref{eq:nabla_h0}.
\end{proof}

This lemma captures a major difference between ordinary spherical gradient, for 
which $\xi \cdot \nabla_0 = 0$, and the $h$-spherical gradient, as seen from our
next proposition and several later results. 

\begin{prop}
For $\xi \in \sph$, we have
\begin{equation} \label{Delta-na-h}
  \Delta_{h,0} = \nabla_{h,0}\cdot \nabla_{h,0} - \xi \cdot \nabla_{h,0}.
\end{equation}
\end{prop}

\begin{proof}
Since $\Delta_h= \nabla_h \cdot \nabla_h$, it follows from the relation \eqref{eq:nabla_h} that
\begin{align} \label{eq:2.2-proof1}
    \Delta_h  =  \frac{d^2}{dr^2} + \frac{1}{r} (\nabla_{h,0}\cdot \xi) \frac{d}{dr} 
          + \xi \frac{d}{dr} \left (\f{1}{r} \nabla_{h,0} \right) + \frac{1}{r^2} \nabla_{h,0} \cdot \nabla_{h,0}. 
\end{align}
By the definition of reflection, $\la \xi \s_v,v \ra = - \la \xi ,v \ra$, it follows that 
$$
  E_v \la \xi,v\ra f(\xi)= \frac{ \la \xi, v\ra f(\xi)- \la \xi \s_v, v\ra f(\xi \s_v)}{\la \xi, v\ra} =
   f(\xi) + \la \xi , v\ra f(\xi \s_v)  = (I + \s_v)f(\xi); 
$$
moreover, by \eqref{eq:nabla}, it follows readily that $\nabla_0 \cdot \xi = d-1$. Consequently, by
\eqref{eq:nabla_h0}, we obtain
\begin{align} \label{nabla_h0x}
   \nabla_{h,0} \cdot \xi & = \nabla \cdot \xi + \sum_{v \in R_+}\k_v  E_v \la \xi,v\ra 
       =  d-1 + \sum_{v \in R_+}  \k_v (I + \s_v) \\
        &  = d-1 + 2 \g_\k - \sum_{v \in R_+} \k_v (I - \s_v) = 2 \lambda_\k + 1 - \xi \cdot \nabla_{h,0} \notag
\end{align}
by \eqref{xi-na_h0}. Furthermore, since $\frac{d}{dr}$ commutes with $\xi \cdot \nabla_{h,0}$, 
$$
\xi \frac{d}{dr} \left (\f{1}{r} \nabla_{h,0} \right) =  - \frac{1}{r^2} \xi \cdot \nabla_{h,0}+
   (\xi \cdot \nabla_{h,0})  \frac{1}{r} \frac{d}{dr}.    
$$
Inserting the last two displayed identities into \eqref{eq:2.2-proof1} shows that 
$$
 \Delta_{h} = \frac{d^2}{dr^2} + \frac{2 \l_\k+1}{r}  \frac{d}{dr} 
   + \frac{1}{r^2} \left(  \nabla_{h,0} \cdot  \nabla_{h,0} - \xi \cdot \nabla_{h,0} \right), 
$$
which proves, upon comparing with \eqref{eq:Delta_h}, the identity \eqref{Delta-na-h}. 
\end{proof}

We now define analogues of the angular derivatives $D_{i,j}$. 

\begin{defn}
For $1 \le i, j \le d$, define differential-difference operators 
$$
  \CD_{i,j} : = x_i \CD_j - x_j \CD_i. 
$$
\end{defn}

By the definition of $D_{i,j}$, we can write 
\begin{equation}\label{decompCDij}
    \CD_{i,j} = D_{i,j} + E_{i,j}, \qquad E_{i,j}:= \sum_{v \in R_+} \k_v (x_i v_j - x_jv_i) E_v.
\end{equation}

\begin{prop} 
For $f, g  \in C^1(\sph)$ and $1 \le i < j \le d$, 
\begin{equation} \label{CDij=CDij}
   \int_{\sph} \CD_{i,j} f(x) g(x) h_\k^2(x) d\s(x) = - \int_{\sph} f(x)  \CD_{i,j} g(x)h_\k^2(x) d\s(x). 
\end{equation}
\end{prop}
 
\begin{proof}
Let us first assume that $\k_v \ge 1$, so that $h_\k(x)$ is continuously differentiable. In
particular, we have then 
$$
  \partial_i \left(g(x) h_\k^2(x)\right) = (\partial_i g(x))  h_\k^2(x) + g(x) 
     \sum_{v \in R_+} 2 \k_v \frac{h_\k^2(x) }{\la x,v\ra} v_i, 
$$
which yields immediately the relation
$$
  D_{i,j} \left(g(x) h_\k^2(x)\right) = D_{i,j} g(x)  h_\k^2(x) + g(x) 
     \sum_{v \in R_+} 2 \k_v \frac{h_\k^2(x) } {\la x,v\ra} (x_i v_j - x_j v_i). 
$$
By the identity \eqref{Dij=Dij}, the differential part of $\CD_{i,j}$ then satisfies 
\begin{align*}
 \int_{\sph} D_{i,j} f(x)  g(x) h_\k^2(x) d\s(x) = & - \int_{\sph}  f(x) D_{i,j} \left(g(x) h_\k^2(x)\right) d\s(x) \\ 
  = & - \int_{\sph}  f(x) D_{i,j} g(x) h_\k^2(x) d\s(x) \\
       & - 2 \int_{\sph}  f(x) g(x) \sum_{v \in R_+} \k_v \frac{x_i v_j - x_j v_i}{\la x,v\ra} h_\k^2(x) d\s(x). 
\end{align*}

Next we consider $E_{i,j}$ term. By the definition, we have 
\begin{align*}
 \int_{\sph} E_{i,j} f(x) g(x) h_\k^2(x) d\s(x) = &  \sum_{v \in R_+} \k_v  \left[ 
    \int_{\sph} (x_i v_j - x_j v_i) \frac{f(x) g(x)}{\la x, v \ra} h_\k^2(x) d\s(x) \right. \\
 & \qquad -  \left.  \int_{\sph} (x_i v_j - x_j v_i) \frac{f(x\s_v) g(x)}{\la x, v \ra} h_\k^2(x) d\s(x)  \right].
\end{align*}
In the last integral in the right hand side, we make a change of variables $x \mapsto x \s_v$. 
Since $h_\k^2(x)$ is invariant under the reflection group, $\la x \s_v, v\ra = \la x, v \s_v \ra = - \la x,v \ra$
by the definition of $\s_v$ and 
$$
  (x \s_v)_i v_j - (x\s_v)_j v_i = x_i v_j -x_j v_i - \frac{2 \la x, v\ra}{\|v\|^2} (v_i v_j - v_jv_i) = x_i v_j -x_j v_i,
$$
we see that  
$$
   \int_{\sph} (x_i v_j - x_j v_i) \frac{f(x\s_v) g(x)}{\la x, v \ra} h_\k^2(x) d\s   
     =  - \int_{\sph} (x_i v_j - x_j v_i) \frac{f(x) g(x \s_v)}{\la x, v \ra} h_\k^2(x) d\s. 
$$
Consequently,  the difference  part of $\CD_{i,j}$ satisfies
\begin{align*}
 \int_{\sph} E_{i,j} f(x) g(x) h_\k^2(x) d\s(x) = &  -  \int_{\sph} f(x) E_{i,j} g(x) h_\k^2(x) d\s(x) \\
    & +2 \sum_{v \in R_+} \k_v  \int_{\sph} f(x) g(x)
     \frac{x_i v_j - x_j v_i}{\la x, v \ra} h_\k^2(x) d\s(x).
\end{align*} 
By \eqref{decompCDij}, summing up the integral relations for $D_{i,j}$ and $E_{i,j}$ proves
\eqref{CDij=CDij}. This completes the proof under the assumption of $\k_v \ge 1$. The general 
case of $\k_v \ge 0$ follows from analytic continuation. 
\end{proof} 

Our next proposition gives an expression of $\nabla_{h,0}$ in terms of $\CD_{i,j}$. 

\begin{prop}\label{lemma1}
The $j$th component of $\nabla_{h,0}$ satisfies
\begin{equation}\label{nabla_h0j}
 (\nabla_{h,0})_j f(\xi) =\sum_{i=1, i \ne j}^d \xi_i \CD_{i,j} f(\xi) + \xi_j (\xi \cdot \nabla_{h,0}),  
\quad \xi \in \sph. 
\end{equation}
\end{prop}

\begin{proof}
For the classical spherical gradient $\nabla_0$, we know that \cite[p. 25]{DaiXbook}
$$
  (\nabla_0)_j f(\xi) = \sum_{i=1, i \ne j}^d \xi_i D_{i,j} f(\xi), \qquad \xi \in \sph.
$$
By \eqref{eq:nabla_h0}, we can then write
$$
    (\nabla_{h,0})_j f(\xi)  = \sum_{i=1, i \ne j}^d \xi_i D_{i,j} f(\xi) + \sum_{v \in R_+} \k_v E_v f(\xi) v_j. 
$$
By the definition of $E_{i,j}$ and setting $E_{i,i}:= 0$, we have 
\begin{align*}
   \sum_{i=1}^d \xi_i E_{i,j} f(\xi) = \sum_{v \in R_+} \k_v E_v f(\xi) \sum_{i=1}^d \xi_i (\xi_i v_j - \xi_j v_i) 
     = \sum_{v \in R_+} \k_v E_v f(\xi) (v_j - \xi_j \la \xi, v\ra), 
\end{align*}   
which leads to, upon rearranging and using the definition of $E_v$, that 
$$     
  \sum_{v \in R_+} \k_v E_v f(\xi) v_j =  \sum_{i=1}^d \xi_i E_{i,j} f(\xi) 
          + \xi_j \sum_{v\in R_+} \k_v (f(\xi) - f(\xi \s_v)).  
$$
Since $\CD_{i,j} = D_{i,j} + E_{i,j}$, \eqref{nabla_h0j} follows from the above relations and
\eqref{xi-na_h0}.
\end{proof}

Our next goal is to derive an adjoint of $\nabla_{h,0}$, for which we need a lemma whose proof
relies on the following identity  (\cite[p. 156]{DX}): For $1 \le i,j \le d$, 
\begin{equation} \label{eq:CDxjf}
  \CD_i (x_j f(x)) = x_j \CD_i f(x) + \delta_{i,j} f(x) + 2 \sum_{v \in R_+} \k_v \frac{v_i v_j}{\|v\|^2} f(x \s_v).
\end{equation}

\begin{lem} \label{lemma2}
For $1 \le i < j \le d$, 
$$
  \sum_{i=1, i \ne j}^d \CD_{i,j} (\xi_j g(\xi)) = (\nabla_{h,0})_j g(\xi) - (\g_\k + d-1) \xi_j g(\xi) 
       - \sum_{v \in R_+} \k_v  (\xi \s_v)_j g(\xi \s_v).
$$
\end{lem}

\begin{proof}
Using the identity in the previous lemma, it is easy to see that 
$$
\CD_{i,j} (\xi_i g(\xi)) = \xi_j \CD_{i,j}g(\xi) - \xi_i g(\xi)+ 2 \sum_{v \in R_+} \k_v \frac{\xi_iv_j-\xi_jv_i}{\|v\|^2} 
  v_j g(\xi \s_v). 
$$
By $\sum_{j=1}^d (\xi_iv_j-\xi_jv_i) v_i = v_j \la \xi ,v\ra- \xi_j \|v\|^2 $ and the definition of 
$\xi \s_v$, we obtain
\begin{align*}
 \sum_{i=1, i\ne j }^d \CD_{i,j} (\xi g(\xi)) = & \sum_{i=1, i\ne j}^d \xi_i D_{i,j} g(\xi) - (d-1) \xi_j g(\xi)
      -  \sum_{v \in R_+} \k_v \left(\xi_j+ (\xi \s_v)_j \right)    g(\xi \s_v).
\end{align*}
Applying the identity \eqref{nabla_h0j} one more time shows that 
\begin{align*}
 \sum_{i=1, i\ne j }^d \CD_{i,j} (\xi g(\xi)) = (\nabla_{h,0})_j g(\xi) - (d-1) \xi_j g(\xi) 
  - \sum_{v\in R_+} \k_v \left (\xi_j g(\xi) + (\xi \s_v)_j  g(\xi \s_v) \right ),
\end{align*} 
which becomes the stated identity if we break the last sum as two sums and recognize that 
the first one is equal to $\g_\k \xi_j g(\xi)$. The proof is completed. 
\end{proof}

We are now in a position to identify the adjoint operator of $\nabla_{h,0}$ in $L^2(\sph, h_\k^2)$, 
which is given in the following integration by parts formula. 

\begin{prop}
For $f, g \in C^1(\sph)$, 
\begin{align} \label{nabla_h0-adjoint}
   & \int_{\sph} \nabla_{h,0} f(\xi) g(\xi) h_\k^2(\xi)d \s(\xi) \\
    & \qquad  = -  \int_{\sph}  f(\xi) \left[\nabla_{h,0} g(\xi) - (2 \l_\k+1)\xi  g(\xi) \right] h_\k^2(\xi)d \s(\xi).  
    \notag
\end{align}
\end{prop}

\begin{proof}
Using the identity \eqref{nabla_h0j} we can decompose the $j$-th component of the left hand side of \eqref{nabla_h0-adjoint} as a sum of two integrals,  
\begin{align*}
  I_1 & =  \int_{\sph}  \sum_{i=1}^d \xi_i \CD_{i,j} f(\xi) g(\xi) h_\k^2(\xi)d \s (\xi), \\ 
  I_2 & =  \int_{\sph}  \sum_{v \in R_+} \k_v \xi_j ( f(\xi) - f(\xi \s_v)) g(\xi) h_\k^2(\xi)d \s (\xi).
\end{align*}
For $I_1$ we use the identity \eqref{CDij=CDij} and Lemma \ref{lemma2} to deduce 
\begin{align*}
  I_1 & =  -  \int_{\sph}  f(\xi)  \sum_{i=1}^d \CD_{i,j} (\xi_i g(\xi))  h_\k^2(\xi)d \s (\xi) \\
        & =   -  \int_{\sph}  f(\xi)  (\nabla_{h,0})_j g(\xi)  h_\k^2(\xi)d \s (\xi) + 
           (\g_k + d-1)  \int_{\sph} \xi_j  f(\xi) g(\xi)  h_\k^2(\xi)d \s (\xi)\\
             &  \quad + \sum_{v \in R_+} \k_v  \int_{\sph} f(\xi) (\xi \s_v)_j g(\xi \s_v)  h_\k^2(\xi)d \s (\xi),        
\end{align*}
whereas for $I_2$, we break the integral as a sum of two integrals and change variable 
$\xi \mapsto \xi \s_v$ in the second one to obtain 
\begin{align*}
  I_2 & =   \int_{\sph}  \sum_{v \in R_+} \k_v \xi_j f(\xi) g(\xi) h_\k^2(\xi)d \s (\xi) - \int_{\sph}  \sum_{v \in R_+} \k_v \xi_i f(\xi \s_v) g(\xi) h_\k^2(\xi)d \s (\xi) \\
     & = \g_\k  \int_{\sph} \xi_j f(\xi) g(\xi) h_\k^2(\xi)d \s (\xi) - \sum_{v \in R_+} \k_v  \int_{\sph} f(\xi) 
         (\xi \s_v)_j g(\xi \s_v)  h_\k^2(\xi)d \s (\xi). 
\end{align*}
Adding the expressions for $I_1$ and $I_2$ proves \eqref{nabla_h0-adjoint}.
\end{proof}

Our next two propositions give further connections between the operators $\CD_{i,j}$, the $h$-spherical 
gradient $\nabla_{h,0}$ and $h$-Laplace-Beltrami operator $\Delta_{h,0}$.  

\begin{prop}
The operators $\CD_{i,j}$ and $\nabla_{h,0}$ satisfy the relation 
\begin{equation} \label{nabla_h0=CDij}
   \nabla_{h,0} f(\xi) \cdot  \nabla_{h,0} g(\xi) + (\xi \cdot \nabla_{h,0} f(\xi) ) (\xi \cdot \nabla_{h,0} g(\xi) )
   = \sum_{1 \le i < j \le d} \CD_{i,j} f(\xi) \CD_{i,j}g(\xi)
\end{equation}
for $\xi \in \sph$. In particular, 
\begin{equation} \label{nabla_h0=CDij2}
   |\nabla_{h,0} f(\xi) |^2  + |\xi \cdot \nabla_{h,0} f(\xi)|^2= \sum_{1 \le i < j \le d} |\CD_{i,j} f(\xi)|^2, \qquad    \xi \in \sph. 
\end{equation}

\end{prop}

\begin{proof}
For simplicity, we denote the right hand side of \eqref{nabla_h0=CDij} by $[\CD f, \CD g]$. 
By the decomposition of $\CD_{i,j}$ in \eqref{decompCDij}, we can write
$$
   [\CD f, \CD g] = [Df, D g] + [Df , Eg]+[Ef, Dg] + [Ef,Eg],
$$
where the brackets in the right hand side is defined in analogy to $[\CD f, \CD g]$. By
\eqref{nabla=Dij}, we have $[Df, Dg] = \nabla_0 f(x) \cdot \nabla_0 g(x)$. By 
the definition of $E_{i,j} f$ and the fact that $x_i v_j - x_j v_i = D_{i,j} \la x, v\ra$, we have
\begin{align*}
  [D f, Eg] =  \sum_{1 \le i < j \le d} D_{i,j} f(\xi) E_{i,j}g(\xi) 
                =   \sum_{v \in R_+} \k_v E_v g(\xi) \sum_{1\le i < j \le d} D_{i,j} f(\xi) D_{i,j} \la \xi, v\ra. 
\end{align*}
Since, by \eqref{eq:nabla}, we have for $x = r\xi$,
\begin{equation} \label{nabla0x.v}
  \nabla_0 \la v, \xi \ra =  \frac{1}{r}  \nabla_0 \la v, x \ra = \nabla \la v ,x \ra - \xi \frac{d}{dr} \la v, r \xi\ra
        = v - \la v, \xi \ra \xi,
\end{equation}
it follows from \eqref{nabla=Dij} and $\xi \cdot \nabla_0 f(\xi) =0$ that 
$$
   \sum_{1\le i < j \le d} D_{i,j} f(\xi) D_{i,j} \la \xi, v\ra = \nabla_0 f(\xi) \cdot \nabla_0 \la v, \xi \ra
       = \nabla_0 f(\xi) \cdot (v - \la v, \xi \ra \xi )
       = \nabla_0 f(\xi) \cdot v.
$$
Consequently,  we conclude that 
$$
   [D f, Eg] =   \sum_{v \in R_+} \k_v  E_v g(\xi) \nabla_0 f(\xi) \cdot v.  
$$ 
Reversing the role of $f$ and $g$, we obtain an analogue expression for $[Ef, Dg]$. Furthermore,
the above consideration also shows that 
\begin{align*}
   \sum_{1 \le i < j \le d} (\xi_i v_j - \xi_j v_i) (\xi_i u_j - \xi_j u_i) 
      & =  \sum_{1 \le i < j \le d} D_{i,j} \la v,\xi\ra D_{i,j} \la u,\xi\ra \\
      & = \nabla_0  \la v,\xi\ra \cdot  \nabla_0  \la u,\xi\ra = \la v, u \ra - \la v,\xi \ra \la u,\xi\ra, 
\end{align*}
so that, by the definition of $E_{i,j}$ and \eqref{xi-na_h0}, 
\begin{align*}
 [Ef, Eg]   = & \sum_{v \in R_+} \sum_{u \in R_+} \k_v \k_u E_v f(\xi) E_u g(\xi) 
     \sum_{1 \le i < j \le d} (\xi_i v_j - \xi_j v_i) (\xi_i u_j - \xi_j u_i) \\
       &  - \sum_{v \in R_+} \sum_{u \in R_+} \k_v \k_u E_v f(\xi) E_u g(\xi)  \la v,\xi \ra \la u,\xi\ra \\
   =  &  \sum_{v \in R_+} \sum_{u \in R_+} \k_v \k_u E_v f(\xi) E_u g(\xi)  \la v, u\ra
       - (\xi \cdot \nabla_{h,0} f(\xi) ) (\xi \cdot \nabla_{h,0} g(\xi) ).
\end{align*}

Summing up these expressions, we deduce then
\begin{align*}
  [\CD, \CD]  = & - (\xi \cdot \nabla_{h,0} f(\xi) ) (\xi \cdot \nabla_{h,0} g(\xi) ) + \nabla_0 f(\xi) \cdot \nabla_0 g(\xi)
     +  \sum_{v \in R_+} \k_v  E_v g(\xi) \nabla_0 f(\xi) \cdot v \\
                       & +  \sum_{v \in R_+} \k_v  E_v f(\xi) \nabla_0 g(\xi) \cdot v
                        + \sum_{v \in R_+} \sum_{u \in R_+} \k_v \k_u E_v f(\xi) E_u g(\xi)  \la v, u\ra.
\end{align*}
By the decomposition \eqref{eq:nabla_h0}, the last four terms in the right hand side of the above expression 
is precisely $\nabla_{h,0} f(\xi) \cdot \nabla_{h,0} g(\xi)$. This completes the proof. 
\end{proof}

The analog of \eqref{Delta=Dij2}, however, takes a must more complicated form.
\begin{prop}\label{prop:Delta=Dij-h}
For $\xi \in \sph$, 
\begin{align}\label{Delta=Dij-h}
  \Delta_{h,0} = & \sum_{1 \le i < j \le d} \CD_{i,j}^2 - (\xi \cdot \nabla_{h,0})^2+ 2 \l_\k \xi \cdot \nabla_{h,0} \\
           &- 2 \sum_{v \in R_+} \k_v^2 (I- \s_v)  + \sum_{v \in  R_+} \k_v^2 (I-\s_v)^2. \notag
\end{align}
\end{prop}

\begin{proof}
We start with the main term in the right hand side. Since $\CD_{i,j} = D_{i,j} + E_{i,j}$, we can write
\begin{equation}\label{pf:2.10-1}
  \sum_{i< j} \CD_{i,j}^2 = \sum_{i<j}  D_{i,j}^2 + \sum_{i<j} D_{i,j} E_{i,j} +\sum_{i<j}  E_{i,j} D_{i,j} + \sum_{i<j} E_{i,j}^2. 
\end{equation} 
The first term in the right hand side, by \eqref{Delta=Dij2}, is equal to $\Delta_0 = \nabla_0 \cdot \nabla_0$. For
the second term, we use the fact that $\xi_i v_j - \xi_j v_i = D_{i,j} \la \xi,v\ra$ and the fact that $D_{i,j}$ is a derivation to
write it as 
\begin{align*}
 \sum_{i<j} D_{i,j} E_{i,j} = & \sum_{v \in R_+} \k_v \sum_{1 \le i< j \le d} \left (D_{i,j}^2 \la \xi, v\ra) E_v
     + D_{i,j}  \la \xi ,v \ra D_{i,j} E_v \right) \\
   = &  \sum_{v \in R_+} \k_v  (\Delta_0 \la \xi, v\ra) E_v +  \sum_{v \in R_+} \k_v  (\nabla_0 \la \xi, v\ra) \cdot 
      \nabla_0 E_v      
\end{align*}
by \eqref{Delta=Dij2} and \eqref{nabla=Dij}. Now, by \eqref{eq:nabla}, it is easy to see that 
$\Delta_0 \la \xi, v\ra = - (d-1) \la \xi, v\ra$, which implies, together with \eqref{nabla0x.v}, that 
\begin{align*} 
 \sum_{i<j} D_{i,j} E_{i,j} =&\, - (d-1)\sum_{v \in R_+} \k_v   \la \xi, v\ra  E_v + 
     \sum_{v \in R_+} \k_v \left( v \cdot \nabla_0  E_v - \la v, \xi \ra \xi \cdot \nabla_0 E_v \right) \notag \\
    = &\, - (d-1) \xi \cdot \nabla_{h,0} +   \sum_{v \in R_+} \k_v v \cdot \nabla_0  E_v     
\end{align*}
upon using \eqref{xi-na_h0} and the fact that $\xi \cdot \nabla_0 = 0$. The third term in the right hand
side of \eqref{pf:2.10-1} is 
$$
  \sum_{i<j}  E_{i,j} D_{i,j}  =  \sum_{v \in R_+} \k_v \sum_{i< j} (\xi_i v_j - \xi_j v_i) E_v D_{i,j}. 
$$
Using the definition of $E_v$ and $(\xi \s_v)_j = \xi_j - 2 \la \xi, v\ra v_j / \|v\|^2$, we see that 
$$
   E_v (\xi_i v_j f(\xi)) = \frac{\xi v_j f(\xi) - v_i (\xi\s_v)_j f(\xi \s_v)}{\la \xi,v\ra}
      = \xi v_j E_v f(\xi) + \frac{2 \la \xi,v\ra}{\|v\|^2} v_i v_j,
$$
which implies that $E_v (\xi_i v_j - \xi_j v_i)=(\xi_i v_j - \xi_j v_i) E_v$. Hence, using \eqref{nabla0x.v}
and \eqref{xi-na_h0}, we conclude that
\begin{align*}
  \sum_{i<j}  E_{i,j} D_{i,j}  & = \sum_{v \in R_+} \k_v E_v  \sum_{i< j} (\xi_i v_j - \xi_j v_i) D_{i,j} 
  = \sum_{v \in R_+} \k_v E_v  \sum_{i< j } D_{i,j} \la \xi ,v\ra D_{i,j} \notag \\ 
 & =   \sum_{v \in R_+} \k_v E_v \nabla_0 \la \xi,v\ra \cdot \nabla_0 
   =  \sum_{v \in R_+} \k_v E_v v \cdot \nabla_0.  
\end{align*}
Finally, the fourth term in the right hand side of \eqref{pf:2.10-1} can be written as 
\begin{align*} 
  \sum_{i<j}  E_{i,j}^2  & = \sum_{v \in R_+} \sum_{u \in R_+} \k_v \k_u 
        \sum_{i<j} (\xi_i v_j - \xi_j v_i) E_v (\xi_i u_j - \xi_j u_i) E_u  \notag \\ 
 & = \sum_{v \in R_+} \sum_{u \in R_+} \k_v \k_u 
        \sum_{i<j} D_{i,j} \la \xi, v\ra  D_{i,j} \la \xi,u \ra E_v E_u  \notag  \\
 &  \,\,   \,\,  +  \sum_{v \in R_+} \sum_{u \in R_+} \k_v \k_u 
        \sum_{i<j} D_{i,j} \la \xi, v\ra  E_v (D_{i,j} \la \xi,u \ra) \sigma_v E_u=: \CE_1 +\CE_2.
\end{align*}
By \eqref{nabla=Dij} and \eqref{nabla0x.v}, it is easy to see that  
$$ 
     \sum_{i<j} D_{i,j} \la \xi, v\ra D_{i,j} \la \xi,u \ra = (\nabla_0 \la \xi, v\ra) \cdot (\nabla_0 \la \xi,u \ra) \
      =  \la v,u\ra - \la v,\xi\ra \la u,\xi\ra, 
$$
which implies, by \eqref{xi-na_h0} that 
\begin{align*}
\CE_1  & = \sum_{v \in R_+} \sum_{u \in R_+} \k_v \k_u (\la v,u\ra - \la v,\xi\ra \la u,\xi\ra) E_v E_u \\
  & = \sum_{v \in R_+} \sum_{u \in R_+} \k_v \k_u \la v,u\ra  E_v E_u -  
   \sum_{v \in R_+} \sum_{u \in R_+} \k_v \k_u \la \xi ,u\ra (I-\s_v) E_u. \notag
\end{align*} 
Directly from the definition, $E_v (D_{i,j} \la \xi,u \ra) = 2(v_i u_j -v_j u_i)/\|v\|^2$, from which it is 
easy to verify that 
$$
  \sum_{i<j} D_{i,j} \la \xi, v\ra  E_v (D_{i,j} \la \xi,u \ra) = 
     \frac{2  \la \xi, v\ra} {\|v\|^2}  \la u,v\ra - 2 \la \xi,u\ra
      = - \la  \xi \s_v, u\ra -  \la \xi, u\ra.
$$
Hence, it follows that 
\begin{align*}
\CE_2 = &  -  \sum_{v \in R_+} \sum_{u \in R_+} \k_v \k_u ( \la  \xi \s_v, u\ra +  \la \xi, u\ra)\s_v E_u.
\end{align*}
Consequently, adding the two expressions for $\CE_1$ and $\CE_2$, we conclude that
\begin{align*}
\sum_{i<j}E_{i,j}^2    = & \sum_{v \in R_+} \sum_{u \in R_+} \k_v \k_u \la v,u\ra  E_v E_u -  
   \sum_{v \in R_+} \sum_{u \in R_+} \k_v \k_u \la \xi ,u\ra E_u \\
     & -     \sum_{v \in R_+} \sum_{u \in R_+} \k_v \k_u   \la  \xi \s_v, u\ra \s_v E_u
\end{align*} 
By the definition of $E_\nu$, it is easy to see that $ \la  \xi \s_u, u\ra\sigma_u E_u = -  (I-\s_u)$ and
$$
 \la  \xi \s_v, u\ra\sigma_v E_u = \s_v - \s_v \s_u = (I - \s_v) - (I-\s_v) (I-\s_u)
$$
if $u\ne v$. It follows that 
\begin{align*}
 & \sum_{v \in R_+} \sum_{u \in R_+} \k_v \k_u   \la  \xi \s_v, u\ra \s_v E_u
   =   - \sum_{v \in R_+} \k_v^2 (I- \s_v) +\sum_{v \ne u} \k_v \k_u  (I-\s_u) \\ 
     & \qquad\qquad\qquad
      - \sum_{v \in R_+} \sum_{u \in R_+} \k_v \k_u  (I-\s_u)(I-\s_v) + \sum_{v\in R_+} \k_v (I-\s_v)^2 \\
  & \qquad\qquad =   \g_\k \xi \cdot \nabla_{h,0} - (\xi \cdot \nabla_{h,0})^2 - 2 \sum_{v \in R_+} \k_v^2 (I- \s_v) 
   +\sum_{v \in  R_+} \k_v^2 (I-\s_v)^2.
\end{align*}
Consequently, we conclude that 
\begin{align*}
 \sum_{i<j} E_{i,j}^2 = & \sum_{v \in R_+} \sum_{u \in R_+} \k_v \k_u \la v,u\ra  E_v E_u 
   - 2\g_\k \xi \cdot \nabla_{h,0} +  ( \xi \cdot \nabla_{h,0})^2 \\
     &+ 2 \sum_{v \in R_+} \k_v^2 (I- \s_v) 
   -\sum_{v \in  R_+} \k_v^2 (I-\s_v)^2.
\end{align*}

Putting these expressions that we derived for the four terms in the right hand side of \eqref{pf:2.10-1}, we conclude 
that 
\begin{align*}
 \sum_{i< j} \CD_{i,j}^2 = &\, \nabla_0\cdot \nabla_0 +     \sum_{v \in R_+} \k_v v \cdot \nabla_0  E_v     
 + \sum_{v \in R_+} \k_v E_v v \cdot \nabla_0 \\
 & +  \sum_{v \in R_+} \sum_{u \in R_+} \k_v \k_u \la v,u\ra  E_v E_u 
    - (2 \l_\k +1) \xi \cdot \nabla_{h,0} + (\xi \cdot \nabla_{h,0})^2 \\
        &+ 2 \sum_{v \in R_+} \k_v^2 (I- \s_v) 
   -\sum_{v \in  R_+} \k_v^2 (I-\s_v)^2.
\end{align*}
On the other hand, by the expression of $\nabla_{h,0}$ in \eqref{eq:nabla_h0}, it follows from 
\eqref{Delta-na-h} that 
\begin{align*} 
  \Delta_{h,0} +   \xi \cdot \nabla_{h,0}
         = &\, \nabla_0 \cdot \nabla_0 + \sum_{v \in R_+} \k_v v \cdot \nabla_0 E_v \\
   & +  \sum_{v \in R_+} \k_v E_v v \cdot \nabla_0 + \sum_{v \in R_+} \sum_{u \in R_+} \k_v \k_u \la v, u \ra E_v E_u. \notag
\end{align*}
Comparing the last two displayed equations proves \eqref{Delta=Dij-h}.  
\end{proof}

It is worth mentioning that $\xi \cdot \nabla_{h,0}$ appears in the previous two propositions. For 
ordinary derivatives, $\xi \cdot \nabla_0 =0$. 

As a final result in this section, we state a relation between the $L^2$ norm of $\nabla_{h,0} f$ and 
$(-\Delta_{h,0})^{\f12} f$, which will be used in the next section. For the classical differential operators,
the $L^2$ norm of $\nabla_0$ and $(-\Delta_{0})^{\f12}$ are equal. The equality, however, holds only
for invariant functions in the weighted setting.

\begin{cor} \label{cor:2.11}
If $f$ is invariant under the associated reflection group, then 
\begin{align*}
  \left \| (-\Delta_{h,0})^{\f12} f \right \|_{L^2(\sph, h_\k^2)}^2 = \, & \| \nabla_{h,0} f \|_{L^2(\sph, h_\k^2)}^2.
\end{align*}
\end{cor}

\begin{proof}
If $f$ is invariant, then $\xi \cdot \nabla f(\xi) =0$ and $E_u f(\xi) =0$. For such functions, the
operator $-\Delta_{h,0}$ is evidently self-adjoint by \eqref{Delta=Dij-h} and the latter also implies
that 
\begin{align*}
    \left \| (-\Delta_{h,0})^{\f12} f \right \|_{L^2(\sph, h_\k^2)}^2 
     = & \frac{1}{\o_d^\k} \int_{\sph}  f(\xi) (-\Delta_{h,0}) f(\xi) h_\k^2(\xi) d\s \\
     = &  - \frac{1}{\o_d^\k} \int_{\sph}   f(\xi) \sum_{1 \le i < j \le d} \CD_{i,j}^2 f(\xi) h_\k^2(\xi) d\s,
\end{align*}
which is precisely $\|\nabla_{h,0}\|_{L^2(\sph, h_\k^2)}^2$, using integration by parts by \eqref{CDij=CDij}.
\end{proof}

\section{Uncertainty principles on weighted space on the unit sphere}
\setcounter{equation}{0}

In this section we study uncertainty principle on the weighted space. We denote by $W_2^1(\sph, h_\k^2)$ 
the weighted Sobolev space 
$$
   W_2^1(\sph, h_\k^2) : = \left \{ f \in L^2(\sph, h_\k^2):   \nabla_{h,0} f \in L^2(\sph, h_\k^2) \right \}.
$$

Since the measure $h_\k^2 d\s$ is no longer invariant under the entire orthogonal group, we can no 
longer expect an inequality that holds with minimal taken over all $e \in \sph$ as in \eqref{UCinequality}.
What we are able to do is to take minimal over all the coordinate plans as shall be seen below. 

For our first result on the uncertainty principle, we consider functions that are invariant under
a reflection group $G$. Recall that the weight function $h_\k$ associated with $G$ is defined 
in \eqref{hk-weight} and the constant $\l_\k$ is defined in \eqref{eq:lambda_k}.  

\begin{thm} \label{thm:uncert-weigth-sph}
Let $f \in W_2^1(\sph, h_\k^2)$ be invariant under the reflection group $G$ and assume 
$\int_{\sph} f(x) h_\k^2(x) d\s =0$ and $\|f\|_{L^2(\sph, h_\k^2)} =1$.  Then  
\begin{equation} \label{eq:uncert-h}
  \min_{1\le i \le d} \frac{1}{\o_d^\k} \int_{\sph} (1-x_i) |f(x)|^2 h_\k^2(x) d\s \, \|\nabla_{0} f\|_{L^2(\sph, h_\k^2)}^2 
        \ge C_{\k,d},
\end{equation}
where 
$$
   C_{\k,d} = 2 \l_\k \left(1 - \frac{ \sqrt{2 \l_\k}}{\sqrt{(\l_\k + \f12)^2 + 2 \l_\k}}\right).
 $$
\end{thm}

\begin{proof}
For convenience, we set
$$
r:=\f 1{\o_d^\k} \int_{\sph} (1-x_i)|f(x)|^2 h_\k^2(x) d\s(x) \quad \hbox{and} \quad
     Lf: =  r \|\nabla_{0} f\|_{L^2(\sph, h_\k^2)}^2.
$$
Our goal is to show that $ L f$ is bounded below by a constant. Since $\|f\|_{L^2(\sph, h_\k^2)} =1$, it is 
evident that $r\in (0,2)$. By \eqref{nabla_h0-adjoint} with $g = f$, we have 
\begin{equation}\label{eq:proof3.1}
 (2\l_\k+1)  \int_{\sph} x |f(x)|^2 h_\k^2(x) d\s = 2  \int_{\sph} f(x)  \nabla_{h,0} f(x)h_\k^2(x)d\s.
\end{equation}
Since $f$ is invariant under the reflection group, $\nabla_{h,0} = \nabla_0$ and 
$x \cdot \nabla_{h,0} f(x) =0$ by \eqref{xi-na_h0}. Hence, for $1 \le i \le d$, 
\begin{align*}
\left |x_i (\nabla_{h,0})_i f(x) \right |^2 =  \sum_{j \ne i}^d \left | x_j (\nabla_{0})_j f(x) \right |^2 
 & \le \sum_{j \ne i} x_j^2  \sum_{j \ne i} \left| (\nabla_{0})_j f(x) \right |^2 \\ 
 & \le (1-x_i^2)  \sum_{j \ne i} \left| (\nabla_{0})_j f(x) \right |^2,
\end{align*}
which implies immediately that 
$$
  \left |(\nabla_{0})_i f(x) \right |^2 = \left |x_i (\nabla_{0})_i f(x) \right |^2 +
    (1-x_i^2)  \left | (\nabla_{0})_i f(x) \right |^2 \le (1-x_i^2)  \left | \nabla_{0}f(x)\right |^2.  
$$
Consequently, applying the Cauchy-Schwartz inequality on the $i$th component of \eqref{eq:proof3.1}
gives
\begin{align*}
  (\l_\k+\tfrac12)^2  |1 - r|^2  & = \frac{1}{\o_d^\k} \left| \int_{\sph} f(x)  (\nabla_{0})_i f(x)h_\k^2(x)d\s \right|^2  \\ 
    & \le \frac{1}{\o_d^\k} \int_{\sph} (1-x_i^2) |f(x)|^2 h_\k^2(x) d\s \|\nabla_{0} f \|_{L^2(\sph, h_\k^2)}^2.
\end{align*}
Since $\|f\|_{L^2(\sph, h_\k^2)}^2 =1$, the Cauchy-Schwartz inequality shows that
\begin{align*}
   \f1{\o_d^\k} \int_{\sph} |f(x)|^2 x_i^2 h_\k^2(x) d\s(x) \ge
         \Bl |\f1{\o_d^\k}\int_{\sph} |f(x)|^2 x_i h_\k^2(x) d\s(x)\Br|^2 = (1-r)^2,
\end{align*}
from which it follows that
\begin{align} \label{proof3.1-2}
      \f1{\o_d^\k}\int_{\sph} |f(x)|^2 (1-x_1^2)h_\k^2(x) d\s(x)\leq 1- (1-r)^2 = (2-r)r.
\end{align}
Consequently, we conclude that
\begin{align*}
   ( \l_\k + \tfrac 12)^2(1-r)^2 \le (2-r) r \|\nabla_0 f\|_{L^2(\sph, h_\k^2)}^2 = (2-r) Lf
\end{align*}
or, equivalently,
\begin{align}  \label{proof3.1-3}
       Lf \ge   ( \l_\k + \tfrac 12)^2 \frac{(1-r)^2}{2-r}.
\end{align}

On the other hand, by \eqref{eigen-eqn} and the assumption that $\int_{\sph} f(x) d\s(x)=0$,
\begin{align*}
 1&  =  \|f\|_{L^2(\sph, h_\k^2)}^2 =  \sum_{n=1}^\infty \|\proj_n f\|_{L^2(\sph, h_\k^2)}^2 \\
     & \le  \frac1 {2 \l_\k} \sum_{n=1}^\infty n(n + 2 \l_\k) \|\proj_n f\|_{L^2(\sph, h_\k^2)}^2  
     = \f{1}{2 \l_\k}\| (-\Delta_{h,0})^{\f12} f\|_{L^2(\sph, h_\k^2)}^2.
\end{align*}
Since $f$ is invariant under the reflection group,  it follows from \eqref{xi-na_h0} and Corollary \ref{cor:2.11}
that $\|(-\Delta_{h,0})^{\f12} f\|_{L^2(\sph, h_\k^2)}^2 = \|\nabla_0 f\|_{L^2(\sph, h_\k^2)}^2$, which implies that 
$Lf  = r  \|\nabla_{0} f\|_{L^2(\sph, h_\k^2)}^2 \ge 2 \l_\k  r$. Together with 
 \eqref{proof3.1-3}, we have shown that 
\begin{align*}
    L f   \ge \max \left\{ (\l_\k  +\tfrac12)^2 \frac {(1-r)^2}{2-r}, 2 \l_\k  r \right\}
       \ge  \min_{t\in (0,2)} \max \left\{ (\l_\k  +\tfrac12)^2 \frac {(1-t)^2}{2-t}, 2 \l_\k t \right\}.
\end{align*}
Since $(1-t)^2 /(2-t)$ is decreasing on $(0,1)$ and increasing on $(1,2)$, it is easy to see that the value of the right
hand side is attained when the two terms are equal, which proves \eqref{eq:uncert-h}. 
\end{proof}

If $f$ is not invariant under the reflection group, the above proof breaks down. Indeed, in
that case, we no longer have $|(\nabla_{h,0})_i f(x)|^2 \le (1-x_i^2) |\nabla_{h,0} f(x)|^2$
since such an inequality would imply, by \eqref{nabla_h0j}, that $x \cdot \nabla_{h,0}(x) = 0$ 
for $x_i = 0$, which does not hold in general.

The optimal constant $C_{\k,d}$ for our uncertainty inequalities are not known; in fact, this constant 
is unknown for the unweighted inequality \eqref{UCinequality}. Following another approach in \cite{DaiX} 
based on orthogonal expansion, it is possible to establish \eqref{eq:uncert-h} in the case of 
$h_\k(x) = \prod_{i=1}^d|x_i|^{\k_i}$ and $G=\ZZ_2^d$, with a better 
constant than the proof of Theorem  \ref{thm:uncert-weigth-sph} could offer; the constant, however, 
is still not optimal. 

\section{Uncertainty principle on the unit ball}
\setcounter{equation}{0}

There is a close relation between analysis on the unit sphere and on the unit ball (see \cite{DX}), which 
can be used to derive uncertainty inequalities on the unit ball from the results in the previous section. 

Let $h_\k(x)$ be the reflection invariant weight function defined in \eqref{hk-weight} for a given
reflection group $G$ and a multiplicity function $\k$. For $\mu \ge 0$, we consider the weight function
\begin{equation} \label{weight-ball}
   W_{\k,\mu} (x):= h_\k^2(x) (1-\|x\|^2)^{\mu - 1/2}, \qquad x \in \BB^d. 
\end{equation}
Evidently $W_{\k,\mu}$ is invariant under the reflection group $G$. When $\k = 0$ or $h_\k(x) =1$, it
becomes the classical rotation invariant weight function
$$
 W_{\mu} (x):= (1-\|x\|^2)^{\mu - 1/2}, \qquad x \in \BB^d. 
$$
Let $L^2(\BB^d, W_{\k,\mu})$ denote the space of measurable functions for which
$$
   \| f\|_{L^2(\BB^d,W_{\k,\mu})} =  \left( b_{\k,\mu} \int_{\BB^d} |f(x)|^2 W_{\k,\mu}(x) dx\right)^{1/2}
$$
are finite, where $b_{\k,\mu}$ denotes the normalization constant so that $\|1\|_{W_{\k,\mu},2} =1$.
For simplicity, let us define $W_2^1(\BB^d, W_{\k,\mu})$ to be the subspace of $L^2(\BB^d,W_{\k,\mu})$ 
for which all first order derivatives of its elements are in $L^2(\BB^d,W_{\k,\mu})$.

The weight function $W_{\k,\mu}$ is known to be closely related to the weight function
$$
      h_{\k,\mu}(x,x_{d+1}) = h_\k^2(x) |x_{d+1}|^{2\mu}, \quad (x,x_{d+1}) \in \SS^d,
$$
which is invariant under the reflection group $G \times \ZZ_2$ acting on $\SS^d$. Let $\SS_+^d$ denote
the north hemisphere defined by $\SS_+^d: = \{(x,x_{d+1}) \in \SS^d: x_{d+1} \ge 0\}$. Then the connection
between the two weight functions is established through 
$$
      (x, x_{d+1}) \in \SS_+^d \mapsto x \in \BB^d, 
$$
which is clearly a bijection. Under this map, it follows that  (cf. \cite[Sec. .8]{DX}) 
\begin{equation} \label{L2normSB}
   \|F \|_{L^2(\SS^d, h_{\k,\mu}^2)} =  \|f\|_{L^2(\BB^d,W_{\k,\mu})}, \qquad \hbox{if} \quad F(x,x_{d+1}) = f(x).
\end{equation}
Recall that $D_{i,j} = x_i \partial_j - x_j \partial_i$ for $i \ne j$. We introduce the notation $\nabla_D$ that
satisfies 
$$
   | \nabla_D f(x)|^2 : = \sum_{1 \le i < j \le d} |D_{i,j} f(x)|^2. 
$$
 
\begin{thm} \label{thm:uncet-ball1}
Let $f \in W_2^1(\BB^d, W_{\k,\mu})$ be invariant under the reflection group $G$ and assume 
$\int_{\BB^d} f(x) W_{\k,\mu}(x) dx =0$ and $\|f\|_{L^2(\BB^d,W_{\k,\mu})} =1$.  Then 
\begin{align} \label{eq:uncert-h-ball}
  \min_{1\le i \le d}  b_{\k,\mu} & \int_{\BB^d} (1-x_i) |f(x)|^2 W_{\k,\mu}(x) dx \, |||\nabla f|||^2 
             \ge C_{\k,\mu}, 
\end{align}
where 
\begin{align} \label{|||nablaf|||} 
|||\nabla f |||^2 :=  \| \nabla f \|_{L^2(\BB^d,W_{\k,\mu+1})}^2 +  \| \nabla_D f \|_{L^2(\BB^d,W_{\k,\mu})}^2
\end{align}
and, with $\l_{\k,\mu} : =  \g_\k + \mu + \f{d-1}{2}$, 
$$
C_{\k,\mu} = 2 \l_{\k,\mu} \left(1 - \frac{ \sqrt{2  \l_{\k,\mu}} }{\sqrt{( \l_{\k,\mu} + \f12)^2 + 2  \l_{\k,\mu}}}\right).
$$
\end{thm}
 
\begin{proof}
Let $F(x, x_{d+1}) = f(x)$ for $x \in \BB^d$. By \eqref{L2normSB}, the uncertainty principle 
\eqref{eq:uncert-h-ball} would follow from the corresponding \eqref{eq:uncert-h} if we could 
establish the identity $||| \nabla f ||| = \| \nabla_0 F\|_{L^2(\sph, h_\k^2)}$. 

For $(y, y_{d+1}) \in \RR^{d+1}$ such that $y_{d+1} \ge 0$, we define a mapping $(y, y_{d+1}) 
\mapsto (r,x)$, where $0< r \le 1$ and $x \in \BB^d$, by 
$$
    y_1 = r x_1,\ldots,  y_d = rx_d, \quad y_{d+1} = r x_{d+1}, \quad x_{d+1} =\sqrt{1-x_1^2-\cdots - x_d^2}.
$$
It is easy to see that this is a bijection and we have 
\begin{align*}
   \frac{\partial}{\partial y_i} & = x_i    \frac{\partial}{\partial r} + \frac{1}{r} \left(
     \frac{\partial}{\partial x_i}  - x_i (x \cdot \nabla^{(x)}) \right ), \quad 1 \le i \le d, \\
  \frac{\partial}{\partial y_{d+1}} & = x_{d+1}  \frac{\partial}{\partial r} + \frac{1}{r} \left( - 
      x_{d+1}(x \cdot \nabla^{(x)} ) \right),
\end{align*} 
where $\nabla^{(x)}$ is the gradient in the variables $x$ in $\RR^d$, which implies, by \eqref{eq:nabla}, 
that 
\begin{align} \label{thm4.2-proof1}
\begin{split}
  (\nabla_0)_i & =  \frac{\partial}{\partial x_i}  - x_i (x \cdot \nabla^{(x)}), \quad 1 \le i \le d, \\
    (\nabla_0)_{d+1} & = -   x_{d+1}(x \cdot \nabla^{(x)} ),
\end{split}
\end{align} 
where $\nabla_0$ is the spherical gradient on $\SS^d$. It follows then that
\begin{align} \label{thm4.2-proof2}
  | \nabla_0 F(x,x_{d+1}) |^2 & = \sum_{i=1}^{d+1}  |(\nabla_0)_i F(x)|^2  \notag \\ 
     & = \sum_{i=1}^d \left ( \frac{\partial f}{\partial x_i}  - x_i (x \cdot \nabla f(x)) \right)^2
       + \left(  x_{d+1}(x \cdot \nabla f(x)) \right)^2  \notag \\ 
       & = |\nabla f(x)|^2 - 2  (x \cdot \nabla f(x))^2 + \|x\|^2  (x \cdot \nabla f(x))^2 + 
           x_{d+1}^2  (x \cdot \nabla f(x))^2 \notag \\
       & =  |\nabla f(x)|^2 - (x \cdot \nabla f(x))^2. 
\end{align}
On the other hand, a simple computation shows that 
\begin{align} \label{thm4.2proof-2}
\sum_{1\le i< j \le d} |D_{i,j} f(x)|^2  & = 
   \frac12 \sum_{i=1}^d \sum_{j=1}^d (x_i \partial_j f(x)- x_j \partial_i f(x))^2  \\
  &    = \|x\|^2 |\nabla f(x)|^2  - (x \cdot \nabla f(x))^2. \notag
\end{align}
Putting these two identities together, we obtain 
$$
 | \nabla_0 F(x,x_{d+1}) |^2 = (1-\|x\|^2)  |\nabla f(x)|^2 +  \sum_{1\le i< j \le d} |D_{i,j} f(x)|^2,
$$
which implies, by \eqref{L2normSB},  that $||| \nabla f ||| = \| \nabla_0 F\|_{L^2(\sph, h_\k^2)}$. The proof
is completed. 
\end{proof}

Let $\CV_n^d(W_{\k,\mu})$ denote the space of orthogonal polynomials of degree $n$ with respect to 
$W_{\k,\mu}$ on $\BB^d$. It is known that $\CV_n^d(W_{\k,\mu})$ is a space of eigenfunctions of a 
second order differential-difference operator $\CD_{\k,\mu}$ defined by (\cite{X01})
$$
   \CD_{\k,\mu}: = \Delta_h -  (x \cdot \nabla)^2 + 2 \l_{\k,\mu} (x \cdot \nabla);
$$
more precisely, we have 
\begin{equation}\label{eigen-eqn2}
   \CD_{\k,\mu} P = - n(n+2 \l_\k) P, \qquad \forall P \in \CV_n^d(W_{\k,\mu}).
\end{equation}
Furthermore, let $\proj_n(W_{\k,\mu})$ be the projection operator 
$$
    \proj_n(W_{\k,\mu}): L^2(\BB^d, W_{\k,\mu}) \mapsto \CV_n^d(W_{\k,\mu}).
$$
Then \eqref{eigen-eqn2} holds for $P =   \proj_n(W_{\k,\mu})f$, which can be used to define the
fractional power of $\CD_{\k,\mu}$. It is known that $\CD_{\k,\mu}$ is related to $\Delta_{h,0}$ 
associated with $h_{\k,\mu}$. In particular, for functions invariant under the reflection group,
we have the relation
$$
   ||| \nabla f ||| = \left \| (-\CD_{\k,\mu})^{1/2} f \right \|_{L^2(\BB^d, W_{\k,\mu})}.
$$ 

In the case of the classical weight function $W_\mu$, there is no need to assume that $f$ is invariant
under a reflection group.   
 
\begin{thm}\label{thm:UCball-Wmu}
For the classical weight function $W_\mu$ with $\mu \ge 0$, let $f \in W_2^1(\BB^d, W_\mu)$ satisfying 
$\int_{\BB^d} f(x) W_{\mu}(x) dx =0$ and $\|f\|_{L^2(\BB^d,W_{\mu})} =1$.  Then 
\begin{align} \label{eq:uncert-ball}
  \min_{e \in \sph}  b_{\mu} & \int_{\BB^d} (1- \la x, e \ra) |f(x)|^2 W_{\mu}(x) dx \, |||\nabla f|||^2 
             \ge C_{\mu}, 
\end{align}
where $C_{\mu} = C_{0,\mu}$, and $||| \nabla f|||$ and $C_{\k,\mu}$ are the same as in the previous theorem. 
\end{thm}

\begin{proof}
Since $W_\mu$ is invariant under the rotation group, it follows that 
$$
   \min_{e \in \sph}  b_{\mu}   \int_{\BB^d} (1- \la x, e \ra) |f(x)|^2 W_{\mu}(x) dx
       =b_{\mu}  \int_{\BB^d} (1- x_1) |f(x\tau)|^2 W_{\mu}(x) dx
$$
for some rotation $\tau \in O(d)$. In the proof of the previous theorem, we have shown, by 
$|||\nabla f ||| = \|\nabla_0 F\|_{L^2(\sph, h_\k^2)}$ and \eqref{thm4.2-proof2}, that 
$$
  |||\nabla f|||^2  =  b_\mu \int_{\BB^d} \left[|\nabla f(x)|^2 - (x \cdot \nabla f(x))^2 \right] W_\mu(x) dx. 
$$
Since $\nabla [f(x \tau)]=  (\nabla f)(x \tau)$ and, by \eqref{eq:nabla}, $x \cdot \nabla = r \frac{d}{dr}$ 
is invariant under the rotation, it follows that $||| f|||$ is rotation invariant. In particular, its value
will not change if we replace $f(x)$ by $f(x \tau)$. Hence, \eqref{eq:uncert-ball} follows from
\eqref{eq:uncert-h-ball}. 
\end{proof}

The geodesic distance function of the sphere $\SS^d$ yields a distance function  $\dist_\BB$ on the unit ball defined by
$$
  \dist_\BB(x,y) := \arccos \left( \la x,y\ra + \sqrt{1-\|x\|^2} \sqrt{1-\|y\|^2} \right), \qquad x,y \in \BB^d. 
$$
It follows, in particular, that if $y \in \BB^d$, then 
$$
    1 - \la x,y\ra  = 1- \cos \dist_\BB(x,y) = \sin^2 \frac{\dist_\BB(x,y)}2 \sim \dist_\BB (x,y)^2. 
$$

The quantity $||| \nabla f |||$ in \eqref{eq:uncert-ball} can be replaced by a more convenient operator norm
that involves only partial derivatives instead of $D_{i,j}$. 

\begin{cor}\label{cor:uncertainty-ball}
Under the assumption of the Theorem \ref{thm:UCball-Wmu}, we have   
\begin{align}\label{UC-ball2}
   & \min_{e \in \sph} \left[ b_\mu \int_{\BB^d} (1- \la x, e\ra ) |f(x)|^2 W_\mu(x) dx \right] \\
    &   \qquad\qquad   \times \left[ b_\mu \int_{\BB^d} \sum_{i=1}^d (1- x_i^2) | \partial_i f(x)|^2 W_\mu(x) dx \right] \ge C_{\mu}. \notag
\end{align}
\end{cor}

\begin{proof}
By the definition of $D_{i,j}$ and the Cauchy-Schwartz inequality, we have
$$
   (D_{i,j} f )^2 = (x_i \partial_j f -x_j \partial_i f)^2 \le 2 \left [ (x_i \partial_j f)^2+ (x_j \partial_i f)^2 \right],
$$
which implies that 
\begin{align*}
\sum_{1 \le i < j \le d}  |D_{i,j} f|^2 &  \le 2 \sum_{1\le i < j \le d}  \left [ (x_i \partial_j f)^2+ (x_j \partial_i f)^2 \right] \\
  &  =  \sum_{i=1}^d \sum_{j=1}^d (x_i \partial_j f)^2  - \sum_{i=1}^d (x_i \partial_i f)^2
      =  \sum_{i=1}^d (\|x\|^2 - x_i^2 ) (\partial_i f)^2. 
\end{align*}
Thus, it follows readily that
$$
      (1-\|x\|^2) |\nabla f|^2 + |\nabla_D f|^2 \le  \sum_{i=1}^d (1 - x_i^2 ) (\partial_i f)^2,
$$
so that \eqref{UC-ball2} follows from \eqref{eq:uncert-ball}. 
\end{proof}

\section{Uncertainty principle on the simplex}
\setcounter{equation}{0}

In this section we use a close relation between analysis on the unit ball and on the simplex 
$$
    \TT^d : = \{x \in \RR^d: x_1 \ge 0,\ldots x_d \ge 0,  1-|x|_1 \ge 0\},
$$ 
where $|x|_1 = x_1 + \ldots + x_d$, to derive uncertainty principles on the simplex $\TT^d$ 
from the results in the previous section. 

Let $h_\k(x)$ be the reflection invariant weight function defined in \eqref{hk-weight} for a given
reflection group $G$ and a multiplicity function $\k$. Assume that $h_\k$ is also invariant under
the group $\ZZ_2^d$ of sign changes. For $\mu \ge 0$ , we consider the weight function \cite{X01}
$$
   U_{\k,\mu} (x):= \frac{h_\k^2(\sqrt{x_1},\ldots \sqrt{x_d})}{\sqrt{x_1 \cdots x_d}} (1-|x|_1)^{\mu - 1/2}, \
   \qquad x \in \TT^d. 
$$
There are essentially two examples of such weight functions, which are given below. 

\medskip\noindent
{\it Example 1. Classical Jacobi weight function}:
\begin{equation}\label{JacobiTd}
    U_{\k,\mu}(t) = \prod_{i=1}^d |x_i|^{\k_i- \f12} (1-|x|_1)^{\mu - \frac12},
\end{equation}
which is associated with the group $G= \ZZ_2^d$.

\medskip\noindent
{\it Example 2. Weight function associated with the hyperoctahedral group}:
\begin{equation}\label{weightTd2}
    U_{\k,\mu}(t) = \prod_{i=1}^d |x_i|^{\k_0- \f12} \prod_{1 \le i < j \le d} |x_i - x_j|^\k 
        (1-|x|_1)^{\mu - \frac12}.
\end{equation}

Let $L^2(\TT^d, U_{\k,\mu})$ denote the space of measurable functions for which
$$
   \| f\|_{L^2(\TT^d,U_{\k,\mu})} =  \left( b_{\k,\mu} \int_{\TT^d} |f(x)|^2 U_{\k,\mu}(x) dx\right)^{1/2}
$$
are finite, where $b_{\k,\mu}$ denotes the normalization constant so that $\|1\|_{U_{\k,\mu},2} =1$.  
The constant $b_{\k,\mu}$ turns out to be the same as the constant for $W_{\k,\mu}$ defined in
\eqref{weight-ball}, as we have 
$$
W_{\k,\mu}(x_1, \ldots,x_d)  =  J(x)^2 U_{\k,\mu}(x_1^2, \ldots x_d^2), 
$$
where $J(x) = \sqrt{x_1 \cdots x_d}$ is the Jacobian of the changing variables 
$$
   \psi: (x_1, \ldots, x_d) \in \BB^d  \mapsto (x_1^2,\ldots, x_d^2) \in \TT^d. 
$$
In fact, under this change of variables, it is known that \cite[Sect. 6.2]{DX}
\begin{equation} \label{L2normBT}
   \|f \|_{L^2(\TT^d, U_{\k,\mu})} =  \|f \circ \psi \|_{L^2(\BB^d,W_{\k,\mu})}. 
\end{equation}
We can also define the Sobolev space $W_2^1(\TT^d, U_{\k,\mu})$ similarly as in the case of $\BB^d$. 

To state our uncertainty principles, let us introduce the differential operators 
$$
   \partial_{i,j} = \partial_i - \partial_j, \qquad 1 \le i < j \le d, 
$$
and define the following functions on $\TT^d$, 
$$
   \varphi_i(x) = \sqrt{x_i (1- |x|_1)}, \quad 1 \le i \le d, \quad \hbox{and} \quad
     \varphi_{i,j} (x) =  \sqrt{x_ix_j}, \quad 1 \le i<j\le d.
$$

\begin{thm}
Let $U_{\k,\mu}$ be the classical Jacobi weight function in \eqref{JacobiTd}. Let 
$f \in W_2^1(\TT^d, U_{\k,\mu})$ satisfying $\int_{\TT^d} f(x) U_{\k,\mu}(x) dx =0$ and 
$\|f\|_{L^2(\TT^d,U_{\k,\mu})} =1$.  Then 
\begin{align} \label{eq:uncert-simplex}
  \min_{1\le i \le d}  b_{\k,\mu} & \int_{\TT^d} (1- \sqrt{x_i}) |f(x)|^2 U_{\k,\mu} (x) dx \, |||\partial f|||^2 
             \ge C_{\k,\mu}, 
\end{align}
where 
\begin{align}  \label{|||partialf|||}
|||\partial f |||^2 :=  \sum_{i=1}^d \| \varphi_i \partial_i f\|_{L^2(\TT^d, U_{\k,\mu})}^2
  + \sum_{1 \le i < j \le d}  \|\varphi_{i,j} \partial_{i,j} f \|_{L^2(\TT^d, U_{\k,\mu})}^2.
\end{align}
and $C_{\k,\mu}$ is equal to $1/4$ of the constant in  Theorem \ref{thm:uncet-ball1}. 
\end{thm}
 
\begin{proof}
Since $f \circ \psi$ is invariant under $\ZZ_2^d$, this theorem follows from Theorem
\ref{thm:uncet-ball1} and \eqref{L2normBT} if we can show that $||| \nabla (f\circ \psi)|||$
in \eqref{|||nablaf|||} is equal to $2 ||| \partial f |||$ in \eqref{|||partialf|||}. 

For $x \in \BB^d$, let $F(x) = f \circ \psi(x) = f(x_1^2,\ldots, x_d^2)$. Then 
$\partial_i F(x) = 2 x_i (\partial_i f)\circ \psi (x)$, so that 
\begin{align*}
   (1-\|x\|^2) |\nabla F(x)|^2 & = 
      (1-\|x\|^2) 4 \sum_{i=1}^d  x_i^2 |(\partial_i f)\circ \psi (x)|^2 
          = 4  \sum_{i=1}^d  \left| (\phi_i \partial_i f ) \circ \psi(x) \right|^2
\end{align*}
and
\begin{align*}
   |\nabla_D F(x)|^2  = 4 \sum_{i=1}^d  x_i x_j |(\partial_{i,j} f)\circ \psi (x)|^2  
         = 4  \sum_{i=1}^d  \left| (\phi_{i,j} \partial_{i,j} f ) \circ \psi(x) \right |^2,
\end{align*}
from which $||| \nabla F||| = ||| \partial f |||$ follows by \eqref{L2normBT}.
\end{proof}

It is worth mentioning that the connection between the sphere $\SS^d$ shows that the 
distance function on the simplex $\TT^d$ inherited from the geometric distance on the
sphere is defined by 
$$
  d_{\TT}(x,y) = \arccos \left(\sqrt{x_1y_1}+ \ldots \sqrt{x_dy_d} + \sqrt{(1-|x|_1)(1-|y|_1)}\right).
$$
In particular, with $e_j$ denoting the $j$th coordinate vector, we see that 
$$
  1- \sqrt{x_j} = 1 - \cos d_\TT(x,e_j) = 2 \sin^2 \frac{d_\TT(x,e_j)}{2} \sim d_\TT(x,e_j)^2.
$$

In the case of weight function associated with the hyperoctahedra group, we can
state an uncertainty principle for symmetric functions, that is, functions invariant
under the symmetric group. 

\begin{thm}
Let $U_{\k,\mu}$ be the weight function in \eqref{weightTd2}. Let $f \in W_2^1(\TT^d, U_{\k,\mu})$ be a symmetric function and satisfy $\int_{\TT^d} f(x) U_{\k,\mu}(x) dx =0$ and $\|f\|_{L^2(\TT^d,U_{\k,\mu})} =1$.
Then the inequality \eqref{eq:uncert-simplex} holds. 
\end{thm}

\medskip \noindent
{\bf Acknowledgement.} 
The author thanks Professor Feng Dai for stimulating discussions on this and other related 
topics, and he thanks two anonymous referees for their careful reading and for pointing out an 
error in the formula (3.15), which led to a substantial revision of the paper.

\bibliographystyle{amsalpha}

\end{document}